\documentclass[12pt]{amsart}
\usepackage{amssymb, amstext, amscd, amsmath}

\usepackage{xy}
\xyoption{all}

\makeatletter
\def\@cite#1#2{{\m@th\upshape\bfseries%
[{#1\if@tempswa{\m@th\upshape\mdseries, #2}\fi}]}} \makeatother
\theoremstyle{plain}
\newtheorem{thm}[subsection]{Theorem}
\newtheorem{cor}[subsection]{Corollary}

\newtheorem{lem}[subsection]{Lemma}
\theoremstyle{definition}

\newtheorem{example}[subsection]{Example}

\newtheorem{defn}[subsection]{Definition}

\newtheorem{eg}[subsection]{Example}

\newcommand{\bC}{{\mathbb{C}}}

\newcommand{\bF}{{\mathbb{F}}}

\newcommand{\bT}{{\mathbb{T}}}

\newcommand{\bZ}{{\mathbb{Z}}}

  \newcommand{\A}{{\mathcal{A}}}
  \newcommand{\B}{{\mathcal{B}}}
  
  \newcommand{\D}{{\mathcal{D}}}
  \newcommand{\E}{{\mathcal{E}}}
  
  \newcommand{\G}{{\mathcal{G}}}
\renewcommand{\H}{{\mathcal{H}}}
  \newcommand{\I}{{\mathcal{I}}}
  \newcommand{\J}{{\mathcal{J}}}
  \newcommand{\K}{{\mathcal{K}}}
\renewcommand{\L}{{\mathcal{L}}}
\newcommand{\M}{{\mathcal{M}}}
  
\renewcommand{\O}{{\mathcal{O}}}

  \newcommand{\T}{{\mathcal{T}}}
  \newcommand{\U}{{\mathcal{U}}}
  \newcommand{\V}{{\mathcal{V}}}
  \newcommand{\W}{{\mathcal{W}}}
  \newcommand{\X}{{\mathcal{X}}}

\newcommand{\fF}{{\mathfrak{F}}}

\newcommand{\fK}{{\mathfrak{K}}}

\newcommand{\rC}{{\mathrm{C}}}

\renewcommand{\phi}{\varphi}
\newcommand{\upchi}{{\raise.35ex\hbox{\ensuremath{\chi}}}}

\newcommand{\qand}{\quad\text{and}\quad}

\newcommand{\qfor}{\quad\text{for}\quad}
\newcommand{\qforal}{\quad\text{for all}\quad}

\newcommand{\AND}{\text{ and }}

\newcommand{\Ad}{\operatorname{Ad}}

\newcommand{\Aut}{\operatorname{Aut}}

\newcommand{\Dim}{\operatorname{dim}}

\newcommand{\ran}{\operatorname{Ran}}
\newcommand{\spn}{\operatorname{span}}

\newcommand{\Ath}{\A_\theta}

\newcommand{\ca}{\mathrm{C}^*}
\newcommand{\cenv}{\mathrm{C}^*_{\text{env}}}

\newcommand{\dlim}{\displaystyle\lim\limits}

\newcommand{\Fn}{\mathbb{F}_n^+}
\newcommand{\Fm}{\mathbb{F}_m^+}
\newcommand{\Fmn}{\mathbb{F}_{mn}^+}
\newcommand{\Fth}{\mathbb{F}_\theta^+}
\newcommand{\Ftheta}{\mathbb{F}_\theta^+}
\newcommand{\Fu}{\mathbb{F}_u^+}

\newcommand{\Fockm}{\ell^2(\Fm)}

\newcommand{\Fockth}{\ell^2(\Fth)}
\newcommand{\lip}{\langle}
\newcommand{\rip}{\rangle}
\newcommand{\ip}[1]{\langle #1 \rangle}
\newcommand{\mt}{\varnothing}
\newcommand{\norm}[1]{\left\| #1 \right\|}
\newcommand{\ol}{\overline}

\newcommand{\ltwo}{\ell^2}

\begin{document}
\title[Dilation Theory for Rank 2 Graph Algebras]{Dilation Theory\\for Rank 2 Graph Algebras}

\author[K.R.Davidson]{Kenneth R. Davidson}
\address{Pure Math.\ Dept.\\U. Waterloo\\
Waterloo, ON\; N2L--3G1\\CANADA}
\email{krdavids@uwaterloo.ca}

\author[S.C.Power]{Stephen C. Power}
\address{Dept.\ Math.\ Stats.\\ Lancaster University\\
Lancaster LA1 4YF \\U.K. }
\email{s.power@lancaster.ac.uk}

\author[D.Yang]{Dilian Yang}
\address{Pure Math.\ Dept.\\U. Waterloo\\
Waterloo, ON\; N2L--3G1\\CANADA}
\email{dyang@uwaterloo.ca}

\begin{abstract}
An analysis  is given of $*$-representations of  rank 2 single
vertex graphs. We develop dilation theory for the non-selfadjoint
algebras $\A_\theta$ and $\A_u$ which are associated with the
commutation relation permutation $\theta$ of a 2 graph and,  more
generally, with commutation relations determined by a unitary
matrix $u$ in $M_m(\bC) \otimes M_n(\bC)$. We show that a defect
free row contractive representation has a unique minimal dilation
to a $*$-representation and we provide a new simpler proof of
Solel's row isometric dilation of two $u$-commuting row
contractions. Furthermore it is shown that the C*-envelope of
$\A_u$ is the generalised Cuntz algebra $\O_{X_u}$ for the product
system $X_u$ of $u$; that for $m\geq 2 $ and $n \geq 2 $
contractive representations of $\Ath$ need not be completely
contractive; and that the universal tensor algebra $\T_+(X_u)$
need not be isometrically isomorphic to $\A_u$.
\end{abstract}

\thanks{First author partially supported by an NSERC grant.}
\thanks{Second author partially supported by EPSRC grant EP/E002625/1.}

\date{}
\maketitle

\section{Introduction}\label{S:intro}

Kumjian and Pask \cite{KumPask} have introduced a family of
C*-algebras associated with
higher rank graphs. In \cite{KP1}, Kribs and Power examined the
corresponding non-selfadjoint operator algebras
and recently Power \cite{P1} has presented a detailed analysis of
the single vertex case, with particular emphasis on rank 2 graphs.
Already this case contains many new and intriguing algebras. In
this paper, we continue this investigation by beginning a study of
the representation and dilation theory of these algebras as well
as more general algebras determined by unitary commutation
relations.

In the $2$-graph case  the C*-algebras are the universal
C*-algebras of unital discrete semigroups which are given
concretely in terms of a finite set of generators and relations of
a special type. Given a permutation $\theta$ of $m \times n$, form
a unital semigroup $\Fth$ with generators
$e_1,\dots,e_m,f_1,\dots,f_n$ which is free in the $e_i$'s and
free in the $f_j$'s, and has the commutation relations $e_i f_j =
f_{j'} e_{i'}$ where $\theta(i,j) = (i',j')$ for $1 \le i \le m$
and $1 \le j \le n$. This is a cancellative  semigroup with unique
factorization \cite{KumPask, P1}.

Consider the left regular representation $\lambda$ of these
relations on $\Fockth$ given by  $\lambda(w) \xi_x = \xi_{wx}$.
The norm closed unital operator algebra generated by these
operators is denoted by $\A_\theta$. In line with Arveson's
approach pioneered in \cite{Arv1}, we are interested in
understanding the completely contractive representations of this
algebra. The message of two recent papers on the Shilov boundary
of a unital operator algebra, Dritschel and McCullough \cite{DMc}
and Arveson \cite{Arv_choq}, is that a representation should be
dilated to a maximal dilation; and these maximal dilations extend
uniquely to $*$-representations of the generated C*-algebra that
factor through the C*-envelope. Thus a complete description of
maximal dilations will lead to the determination of the
C*-envelope.

Kumjian and Pask define a $*$-representation of the semigroup
$\Fth$ to be a representation $\pi$  of $\Fth$ as isometries with
the following property which we call the \textit{defect free}
property:
\[
 \sum_{i=1}^m \pi(e_i)\pi(e_i)^* = I
 = \sum_{j=1}^n \pi(f_j) \pi(f_j)^* .
\]
The universal C*-algebra determined by this family of
representations is denoted $\ca(\Fth)$. We shall show that every
completely contractive representation of $\Ath$ dilates to a
$*$-representation. This allows us in particular to deduce that
the C*-envelope of $\Ath$ is $\ca(\Fth)$. This identification is
due to Katsoulis and Kribs \cite{KK} who show, more generally, that
the universal C*-algebra of a higher rank graph $(\Lambda , d)$ is
the enveloping C*-algebra of the associated left regular
representation algebra $\A_\Lambda$.

The left regular representation of $\Fth$ is not a
$*$-representation. It is important though that it dilates (in
many ways) to a $*$-representation.

A significant class of representations which play a key role in
our analysis are the \textit{atomic} $*$-representations. These
row isometric representations have an orthonormal basis which is
permuted, up to unimodular scalars, by each of the generators.
They have a rather interesting structure, and in a sequel to this
paper \cite{DPYatomic}, we shall completely classify them in terms
of families of explicit partially isometric representations. In
this paper, we see the precursors of that analysis. The dilation
theory for partial isometry representations that we develop will
be crucial to our later analysis.

These atomic representations also allow us to
describe the C*-algebra $\ca(\Fth)$.
Such a description relies on an understanding of the
Kumjian--Pask aperiodicity condition.
The periodic case is characterized in \cite{DYperiod},
leading to the structure of  $\ca(\Fth)$.

An important tool for us will be Solel's generalisation of Ando's
dilation theorem to the case of a pair of row contractions $[A_1
\dots A_m],$  $[B_1 \dots B_n]$ that satisfy the commutation
relations
\[
A_iB_j = \sum_{i'=1}^m\sum_{j'=1}^n u_{(i,j),(i',j')}B_{j'}A_{i'}
\]
where $u = u_{(i,j),(i',j')}$ is a unitary matrix in
$M_{mn}(\bC)$. Solel obtained this result as part of his analysis
of the representation theory for the tensor algebra $\T_+(X)$
associated with a product system  of correspondences $X$. We
obtain a new simple proof  which is based on the
Frahzo-Bunce-Popescu dilation theory of row contractions and the
uniqueness of minimal dilations.

The relevant tensor algebra, as defined in \cite{SolCP}, arises as
a universal algebra associated with a product system of
correspondences,
\[
X_u = \{E_{k,l}=(\bC^m)^{\otimes k}\otimes (\bC^n)^{\otimes l}:k,l
\in \bZ_+\}, \]
 where the composition maps
\[
E_{k,l}\otimes E_{r,s} \to E_{k+r,l+s}
\]
are unitary equivalences determined naturally by $u$. An
equivalent formulation which fits well with our perspectives is to
view $\T_+(X_u)$ as the universal operator algebra for a certain
class of representations (row contractive ones) of the norm closed
operator algebra $\A_u$ generated by creation operators $\lambda
({e_i}), \lambda ({f_j})$ on the Fock space of $X_u$. These
unitary relation algebras generalise the 2-graph algebras
$\A_\theta$. While the atomic representation theory of these
algebras remains to be exposed we can analyse C*-envelopes,
C*-algebra structure and dilation theory in this wider generality
and so we do so. Also we prove, as one of the main results,  that
a defect free row contractive representation of $\A_u$ has a
unique minimal row isometric defect free representation.

Prior to Solel's study \cite{SolCP}, the operator algebra theory
of product systems centered on C*-algebra considerations. In
particular Fowler \cite{Fow}, \cite{Fow2} has  defined and
analyzed the Cuntz algebras $\O_X$ associated with a discrete
product systems $X$ of finite dimensional Hilbert spaces. Such an
algebra is the universal C*-algebra for certain
$*$-representations satisfying the defect free property. We shall
prove that the C*-algebra envelope of $\A_u$ is $\O_{X_u}$.

The atomic representations of the 2-graph semigroups $\Fth$ give
many insights to the general theory. For example we note contrasts
with the representation theory for the bidisc algebra, namely that
row contractive representations of $\A_u$ need not be contractive,
and that contractive representations of $\A_u$ need not be
completely contractive.

We remark that the structure of automorphisms of the algebras
$\A_u$ and a classification up to isometric isomorphism has been
given in \cite{PowSol}. In fact we make use of such automorphisms
and the failure of contractivity of row contractive
representations to show that  $\T_+(X_u)$ and $\A_u$ may fail to
be isometrically isomorphic.

\section{Two-graphs, semigroups and representations}

Let $\theta \in S_{m \times n}$ be a permutation of $m \times n$.
The semigroup $\Fth$ is generated by $e_1,\dots,e_m$ and
$f_1,\dots,f_n$.
The identity is denoted as $\mt$.
There are no relations among the $e$'s, so they generate
a copy of the free semigroup on $m$ letters, $\Fm$;
and there are no relations on the $f$'s,
so they generate a copy of $\Fn$.
There are \textit{commutation relations} between the $e$'s
and $f$'s given by
\[ e_i f_j = f_{j'} e_{i'} \quad\text{where } \theta(i,j) = (i',j') .\]

A word $w\in\Fth$ has a fixed number of $e$'s and $f$'s regardless
of the factorization; and the \textit{degree} of $w$ is $(k,l)$ if
there are $k$ $e$'s and $l$ $f$'s.
The \textit{length} of $w$ is $|w| = k+l$.
The commutation relations allow any word $w\in\Fth$
to be written with all $e$'s first, or with all $f$'s first, say
$w = e_uf_v = f_{v'}e_{u'}$.
Indeed, one can factor $w$ with any prescribed pattern of
$e$'s and $f$'s as long as the degree is $(k,l)$.
It is straightforward to see that the factorization is
uniquely determined by the pattern and that $\Ftheta$
has the unique factorization property.
See also \cite{KumPask,KP1,P1}.

We do not need the notion of a $k$-graph $(\Lambda, d)$, in which
$\Lambda$ is a countable small category with functor $d: \Lambda
\to \bZ^k_+$ satisfying a unique factorisation property. However,
in the single object (i.e. single vertex) rank $2$ case, with
$d^{-1}(1,0), d^{-1}(0,1)$ finite, the small category $\Lambda$,
viewed as a semigroup, is isomorphic to $\Fth$ for some $\theta$
and $d$ is equal to the degree map.

\begin{example}\label{E:favourite}
With $n=m=2$ we note that the relations
\begin{alignat*}{2}
e_1f_1 &= f_2e_1, &\qquad e_1f_2 &= f_1e_2 \\
e_2f_1 &= f_1e_1, &\qquad e_2f_2 &= f_2e_2.
\end{alignat*}
are those arising from the permutation $\theta$ in $S_4$ which is
the $3$-cycle $((1,1),(1,2),(2,1))$. We refer to $\Ftheta$ as the
forward 3-cycle semigroup. The reverse $3$-cycle semigroup is the
one arising from the $3$-cycle \\
$((1,1),(2,1),(1,2))$.
\end{example}

It can be shown that the 24 permutations of $S_4$
give rise to $9$ isomorphism classes of semigroups
$\Ftheta$, where we allow isomorphisms to exchange
the $e_i$'s for $f_j$'s.
The forward and reverse $3$-cycles give non-isomorphic
semigroups \cite{P1}.

\begin{example}\label{E:flip}
With $n=m=2$  the relations
\begin{alignat*}{2}
e_1f_1 &= f_1e_1, &\qquad e_1f_2 &= f_1e_2 \\
e_2f_1 &= f_2e_1, &\qquad e_2f_2 &= f_2e_2.
\end{alignat*}
are those arising from the $2$-cycle permutation $((1,2),(2,1))$.
We refer $\Ftheta$ in this case as the flip semigroup and
$\A_\theta$ as the flip algebra. The generated C*-algebra is
identified in Example \ref{flipCenv} and an illuminating atomic
representation is given in Example \ref{flipdefect}.
\end{example}

Consider the left regular representation $\lambda$ of these
relations. This is defined on $\Fockth$ with the orthonormal basis
$\{ \xi_x : x \in \Fth\}$ by  $\lambda(w) \xi_x = \xi_{wx}$. The
norm closed unital operator algebra generated by these operators
is denoted by $\A_\theta$.

\begin{defn}\label{repsdef}
A \text{representation} of $\Fth$ is a semigroup homomorphism
$\sigma: \Ftheta\to \B(\H)$. If it extends to a continuous
representation of the algebra $\Ath$, then it is said to be
\textit{contractive} or \textit{completely contractive} if the
extension to $\Ath$ has this property.

A representation of $\Ftheta$ is \textit{partially isometric} if
the range consists of partial isometries on the Hilbert space $\H$
and is \textit{isometric} if the range consists of isometries.

A partially isometric representation is \textit{atomic}
if there is an orthonormal basis which is permuted,
up to scalars, by each partial isometry.
That is, $\pi$ is atomic if there is a basis  $\{\xi_k : k\ge1\}$
so that for each $w \in \Fth$, $\pi(w) \xi_k = \alpha \xi_l$ for
some $l$ and some $\alpha \in \bT \cup \{0\}$.

A representation $ \sigma$ is \textit{row contractive} if
$[\sigma(e_1) \dots \sigma(e_m)]$ and
\\ $[\sigma(f_1) \dots
\sigma(f_n)]$ are row contractions, and is \textit{row isometric}
if these row operators are isometries. A row contractive
representation is \textit{defect free} if
\[
 \sum_{i=1}^m \sigma(e_i) \sigma(e_i)^* = I
 = \sum_{j=1}^n \sigma(f_j) \sigma(f_j)^*.
\]
A row isometric defect free representation is called a
\textit{$*$-repre\-sent\-ation} of $\Fth$.
We reserve the term defect free for row contractive representations.
\end{defn}

The row isometric condition is equivalent to saying that the
$\sigma(e_i)$'s are isometries with pairwise orthogonal range;
and the same is true for the $\sigma(f_j)$'s.
In a defect free, isometric representation, the $\sigma(e_i)$'s
generate a copy of the Cuntz algebra $\O_m$
(respectively the $\sigma(f_j)$'s generate $\O_n$)
rather than a copy of the Cuntz--Toeplitz algebra $\E_m$
(resp. $\E_n$) as is the case for the left regular representation.
The left regular representation $\lambda$ is
row isometric, but is not defect free.

There is a universal C*-algebra $\ca(\Fth)$ which can be described
by taking a direct sum $\pi_u$ of all $*$-representations on a
fixed separable Hilbert space, and forming the C*-algebra
generated by $\pi_u(\Fth)$. It is the unique C*-algebra generated
by a $*$-representation of $\Fth$ with the property that given any
$*$-representation $\sigma$, there is a $*$-homomorphism
$\pi:\ca(\Fth) \to \ca(\sigma(\Fth))$ so that $\sigma = \pi
\pi_u$. This C*-algebra is a higher rank graph C*-algebra in the
sense of Kumjian and Pask \cite{KumPask} for the rank two single
vertex graph determined by $\theta$.

\begin{eg}\label{3a_repn} \textbf{Type 3a representations.}
We now define an important family of atomic $*$-representations of
$\Ftheta$. The name refers to the classification obtained in
\cite{DPYatomic}.

Start with an arbitrary \textit{infinite word} or \textit{tail}
$\tau = e_{i_0}f_{j_0}e_{i_1}f_{j_1} \dots$.
Let $\G_s = \G := \Ftheta$, for $s=0,1,2,\dots $, viewed as a discrete set
on which the generators of $\Ftheta$ act as injective maps
by right multiplication, namely,
\[ \rho(w)g = gw \qforal g \in \G. \]
Consider $\rho_s = \rho(e_{i_s}f_{j_s})$ as a map from
$\G_s$ into $\G_{s+1}$.
Define $\G_\tau$ to be the injective limit set
\[
 \G_\tau = \lim_{\rightarrow} (\G_s, \rho_s ) ;
\]
and let $\iota_s$ denote the injections of $\G_s$ into $\G_\tau$.
Thus $\G_\tau$ may be viewed as the union of $\G_0, \G_1, \dots $
with respect to these inclusions.

The left regular action $\lambda$ of $\Fth$ on itself induces
corresponding maps on $\G_s$ by  $\lambda_s(w) g = wg$.
Observe that $\rho_s \lambda_s(w) = \lambda_{s+1}(w) \rho_s$ .
The injective limit of these actions is an action $\lambda_\tau$
of $\Fth$ on $\G_\tau$.
Let $\lambda_\tau$ also denote the corresponding representation of $\Fth$
on $\ltwo(\G_\tau)$.
Let  $\{ \xi_g : g \in \G_\tau\}$ denote the basis.
A moment's reflection shows that this provides a defect free,
isometric representation of $\Fth$; i.e.\ it is a $*$-representation.
\end{eg}

Davidson and Pitts \cite{DP1} classified the atomic
$*$-representations of $\Fm$ and showed that the irreducibles fall
into two types, known as ring representations and infinite tail
representations.
The 2-graph situation analysed in \cite{DPYatomic} turns out to be
considerably more complicated and in particular it is shown that
the irreducible atomic $*$-representations of $\Ftheta$ fall into
six types.
\bigskip

 We now define the more general unitary relation algebras
$\A_u$ which are associated with a unitary matrix $u =
(u_{(i,j),(k,l)})$ in $M_{mn}(\bC)$. Also we define the
(universal) tensor algebra $\T_+(X_u)$ considered by Solel
\cite{SolCP} and the generalised Cuntz algebra $\O(X_u)$, both of
which are associated with a product system $X_u$ for $u$.

Let $e_1,\dots,e_m$ and $f_1,\dots,f_n$ be viewed as bases for the
vector spaces $E = \bC^m$ and $F = \bC^n$ respectively. Then $u$
provides an identification $ u :  E\otimes F \to F\otimes E$ such
that
\[ e_i\otimes f_j = \sum_{i'=1}^m
\sum_{j'=1}^n u_{(i,j),(i',j')} f_{j'}\otimes e_{i'}
\]
or, equivalently,
\[
f_l\otimes e_k=\sum_{i=1}^m\sum_{j=1}^n
\bar{u}_{(i,j),(k,l)}e_i\otimes f_j.
\]
Moreover, for each pair $(k,l)$ in $\bZ_+^2$ with $k+l=r$,
 $u$ determines an unambiguous identification $G_1\otimes \dots \otimes G_r \to
H_1\otimes \cdots \otimes H_r$, whenever each $G_i$ and $H_i$ is
equal to $E$ or $F$ and is such that the multiplicity of $E$ and
$F$ in each product is $k$ and $l$ respectively. Thus these
different patterns of multiple tensor products of $E$ and $F$ are
identified with $E^{\otimes k}\otimes F^{\otimes l}$. The family
$X_u = \{E^{\otimes k}\otimes F^{\otimes l}\}$ together with the
associative multiplication $ \otimes $ induced by $u$, as above,
is an example of a \textit{product system over } $\bZ_+^2$,
consisting of finite dimensional Hilbert spaces.

Let $\H_u $ be the $\bZ_+^2$-graded Fock space $
\sum_{k=0}^\infty\sum_{l=0}^\infty  \oplus (E^{\otimes k}\otimes
F^{\otimes l})$ with the convention $E^{\otimes 0} = F^{\otimes 0}
= \bC$.
The left creation operators $L_{e_i}$, $L_{f_j}$ are defined on
$\H_u$ in the usual way. Thus
\[
L_{f_i}(e_{i_1}\otimes \cdots \otimes e_{i_k}\otimes
f_{j_1}\otimes \cdots \otimes f_{j_l}) = f_i\otimes(e_{i_1}\otimes
\cdots \otimes e_{i_k}\otimes f_{j_1}\otimes \cdots \otimes
f_{j_l}). \]

As in \cite{PowSol} we define the \textit{unitary relation
algebra} $\A_u$ to be the norm closed algebra generated by these
shift operators. Note that for $\Fth$ we have $\Ath = \A_u$ where
the unitary is  the permutation matrix $u$ with $u_{(i,j),(i',j')}
= 1$ if $\theta(i,j)=(i',j')$ and $u_{(i,j),(i',j')}=0$ otherwise.
In consistency with the notation for the left regular
representation of $\Fth$ we shall write $\xi_{e_uf_v}$ for the
basis element $e_{i_1}\otimes \cdots \otimes e_{i_k}\otimes
f_{j_1}\otimes \cdots \otimes f_{j_l}$ where (with tolerable
notation ambiguity) $u = i_1\dots i_k$ and $ v = j_1\dots j_l$.

We define $\bF_u^+$ to be the semigroup generated by the left
creation operators. Moreover we are concerned with representations
of this semigroup that satisfy the unitary commutation relations,
that is, with representations that extend to the complex algebra
$\bC[\bF_u^+]$ generated by the creation operators. This will be
an implicit assumption henceforth.
 Thus a unital representation $\sigma$  of $\Fu$ is determined by
two row operators $A= [A_1\dots A_m],$  $B=[B_1\dots B_n]$ that
satisfy the commutation relations
\[
A_iB_j = \sum_{i'=1}^m\sum_{j'=1}^n u_{(i,j),(i',j')}B_{j'}A_{i'}.
\]
The terms row contractive, row isometric, and partially isometric
are defined as before, and we say that $\sigma$ is contractive or
completely contractive if the extension of $\sigma$ to $\A_u$
exists with this property.

In \cite{SolCP}, Solel defines the universal non-selfadjoint tensor
algebra $\T_+(X)$ of a general product system $X$ of
correspondences. In the present context it is readily identifiable
with the universal operator algebra for the family of row
contractive representations $\pi_{A,B}$ and we take this as the
definition of the  \textit{tensor algebra} $\T_+(X_u)$.

On the C*-algebra side the \textit{generalised Cuntz algebra}
$\O_X$ associated with a product system $X$ is the universal
algebra for a natural family $*$-representation of $X$. See
\cite{Fow}, \cite{Fow2}, \cite{FowRae}. In the present context
this C*-algebra is the same as the universal operator algebra for
the family of defect free row isometric representations
$\pi_{S,T}$ and we take this as the definition of $\O_{X_u}$.

 We shall not
need the general framework of correspondences, for which the
associated C*-algebras are the Cuntz-Pimsner algebras. See
\cite{Rae} for an overview of this. However, let us remark that
the direct system $X_u$ is a direct system of correspondences over
$\bC$. The universality in \cite{SolCP} entails that $\T_+(X_u)$
is the completion of $\bC[\Fu]$ with respect to representations
$\pi_{A,B}$ for which each restriction $\pi_{A,B}|E^{\otimes
k}\otimes F^{\otimes l}$ is completely contractive with respect to
the matricial norm structure arising from the left regular
inclusions $E^{\otimes k}\otimes F^{\otimes l} \subseteq \A_u$.
These matricial spaces are row Hilbert spaces and so, taking
$(k,l) = (1,0)$ and $(0,1)$ we see that $A$ and $B$ are
necessarily row contractions. This necessary condition is also
sufficient. Indeed,  each restriction $\pi_{A,B}|E^{\otimes
k}\otimes F^{\otimes l}$ is  determined by a single row
contraction $[T_1 \dots T_N]$ (which is a tensor power of $A$ and
$B$) and these maps, which are of the form
\[
(\alpha_1, \dots ,\alpha_N) \to [\alpha_1T_1, \dots ,\alpha_NT_N],
\]
are completely contractive.

\begin{example}\label{tail_u}
We now show that as in the case of the permutation algebras
$\A_\theta$, the algebra $\A_u$ has a defect free row isometry
representation $\lambda_\tau$ associated with each infinite tail
$\tau$. In particular there are nontrivial $*$-representations
(Cuntz representations) for the product system $X_u$ and
$\O_{X_u}$ is nontrivial.

Consider, once again, an \textit{infinite word} or \textit{tail}
$\tau = e_{i_0}f_{j_0}e_{i_1}f_{j_1} \dots$.  Let $\H_t = \H_u$,
for $t=0,1,2,\dots ,$ and for $s = 0,1, \dots ,$ define isometric
Hilbert space injections $\rho_s : \H_{s} \to \H_{s+1}$ with $
\rho_{s}(\xi) = \xi \otimes e_{i_s}f_{j_s}$ for each $\xi \in
E^{\otimes k}\otimes F^{\otimes l}$ and all $k, l$.  Let $\H_\tau $
be the Hilbert space $ \dlim_\to \H_s$,  with each $\H_s$
identified as a closed subspace and let $\lambda_\tau$ denote the
induced isometric representation of $\bF_u^+$ on $\H_\tau$.

It follows readily that $\lambda_\tau$ is a row isometric
representation. Moreover, it is a $*$-representation of $\bF_u^+$,
that is, $\lambda_\tau$ has the defect free property. To see this,
let $\xi^{s}_{e_uf_v} $ denote the  basis element of $\H$ equal to
$\xi_{e_uf_v}$ in $\H_{s}$ where $e_u$ and $f_v$ are words as
before with lengths $|u|=k \geq 0, |v|= l \geq 0$. Then
$\xi^{s}_{e_uf_v} = \xi^{s+1}_{e_uf_ve_{i_s}f_{j_s}}$. The
commutation relations show that this vector lies both in the
subspace of $\H_{s+1}$ spanned by the spaces
$\lambda_\tau(e_i)E^{k}\otimes F^{l+1}$, $i=1,\dots ,m$, and in
the subspace spanned by the spaces
$\lambda_\tau(f_j)E^{k+1}\otimes F^{l}$, $j=1,\dots ,n$. It
follows that the range projections of the isometries
$\lambda_\tau(e_i)$, and also those of $\lambda_\tau(f_j)$, sum to
the identity.
\end{example}

\section{$\ca(\A_u)$ and the C*-envelope}

There are three natural C*-algebras associated with $\A_u$ namely
the generated C*-algebra $\ca(\A_u)$, the universal C*-algebra
$\O_{X_u}$, and the C*-envelope $\cenv(\A_u)$. By its universal
property the latter algebra is the smallest C*-algebra containing
$\A_u$ completely isometrically. In the case of $\A_\theta$ the
generated C*-algebra is simply the C*-algebra generated by the
left regular representation of the semigroup $\Fth$.

In this section we show that $\cenv(\Ath) = \ca(\Fth) $ and more
generally that $\cenv(\A_u) = \O_{X_u}$. Also we analyse ideals
and show how this algebra is a quotient of $\ca(\A_u)$.

\begin{lem}\label{cisom}
Let $\lambda_\tau$ be any type {\em 3a} representation of $\Fth$.
Then the imbedding of $\Ath$ into $\ca(\lambda_\tau(\Fth))$ is a
complete isometry. Also, if $\lambda_\tau$ is a tail
representation of $\A_u$ then the imbedding of $\A_u$ into
$\ca(\lambda_\tau(\A_u))$ is a complete isometry.
\end{lem}

\begin{proof}
Let $\A$ be the norm closed subalgebra of
$\ca(\lambda_\tau(\Fth))$ generated by $\lambda_\tau(\Fth)$. We
showed in Example~\ref{3a_repn} that $\lambda_\tau$ is an
inductive limit of copies of $\lambda$. That is, $\ltwo(\G_\tau)$
is the closure of an increasing union of subspaces $\ltwo(\G_s)$,
each is invariant under $\A$, and the restriction of
$\lambda_\tau$ to $\ltwo(\G_s)$ is unitarily equivalent to
$\lambda$. The norm of any matrix polynomial is thus determined by
its restrictions to these subspaces, and the norm on each one is
precisely the norm in $\Ath$. It follows that $\A$ is completely
isometrically isomorphic to $\Ath$. The same argument applies to a
tail representation of the unitary relation algebra $\A_u$.
\end{proof}

\begin{cor}\label{C:quotient}
There is a canonical quotient map from $\ca(\Fth)$ onto
$\ca(\lambda_\tau(\Fth))$ and, more generally, from $\O_{X_u}$
onto $\ca(\lambda_\tau(\A_u))$. Also there is a canonical quotient
map from  $\ca(\lambda_\tau(\A_u))$ onto $\cenv(\A_u)$.
\end{cor}

\begin{proof}
That there are canonical quotient maps from the universal
C*-algebras  $\ca(\Fth)$ and $\O_{X_u}$ follows from the fact that
$\lambda_\tau$ is a $*$-representation.

By Lemma~\ref{cisom},
$\A_u$ imbeds completely isometrically in
$\ca(\lambda_\tau(\A_u))$. Hence there is a canonical quotient
map of $\ca(\lambda_\tau(\A_u))$ onto $\cenv(\A_u)$ which is the
identity on $\A_u$.
\end{proof}

\subsection{Gauge automorphisms}
First we consider the graph C*-algebra $\ca(\Fth)$. It will be
convenient in this subsection to consider a faithful
representation $\pi$, or equivalently a $*$-representation $\pi$
of $\Fth$, so that $\ca(\Fth) = \ca(\pi(\Fth))$. The universal
property of $\ca(\Fth)$ yields a family of \textit{gauge
automorphisms} $\gamma_{\alpha,\beta}$ for $\alpha, \beta \in \bT$
determined by
\[
 \gamma_{\alpha,\beta}(\pi(e_i)) = \alpha \pi(e_i) \qand
 \gamma_{\alpha,\beta}(\pi(f_j)) = \beta \pi(f_j) .
\]
Integration around the 2-torus yields a faithful expectation
\[ \Phi(X) = \int_{\bT^2}  \gamma_{\alpha,\beta}(X) \,d\alpha\,d\beta .\]
It is easy to check on monomials that the range is spanned by
words of degree $(0,0)$ (where $\pi(e_i)^*$ and $\pi(f_j)^*$ count
as degree $(-1,0)$ and $(0,-1)$ respectively).

Kumjian and Pask identify this range as an AF C*-algebra. In our
case, the analysis is simplified.  To recap, the first observation
is that any monomial in $e$'s, $f$'s and their adjoints can be
written with all of the adjoints on the right. Clearly the row
isometric condition means that
\[ \pi(f_i)^*\pi(f_j) = \delta_{ij} = \pi(e_i)^*\pi(e_j) .\]
Also,  observe that if $f_j e_k = e_{k'}f_{j_k}$, for $1 \le k \le
m$, then
\begin{align*}
 \pi(e_i)^*\pi(f_j) &= \pi(e_i)^*\pi(f_j)
 (\sum_k \pi(e_k)\pi(e_k)^*)\\
 &= \sum_k \pi( e_i)^* \pi(e_{k'})\pi(f_{j_k})\pi(e_k)^*
 = \sum_k \delta_{ik'} \pi(f_{j_k})\pi(e_k)^*.
\end{align*}
So, in the universal representation,  every word in the generators
and their adjoints  can be expressed as a sum of words of the form
$xy^*$ for $x,y \in \Fth$.

Next,  observe that for each integer $s \ge1$, the words $W_s$ of
degree $(s,s)$ determine a family of degree $(0,0)$ words,
namely\\ $\{ \pi(x)\pi(y)^* : x,y \in W_s \}$. It is clear that
\[ \pi(x_1)\pi(y_1)^* \pi(x_2)\pi(y_2)^* = \delta_{y_1,x_2}
\pi(x_1)\pi(y_2)^* .\] Thus these operators form a family of
matrix units that generate a unital copy $\fF_s$ of the matrix
algebra $M_{(mn)^s}(\bC)$. Moreover, these algebras are nested
because the identity
\[
 \pi(x)\pi(y)^* =
 \pi(x) \sum_i \pi(e_i)\pi(e_i)^* \sum_j \pi(f_j)\pi(f_j)^*\ \pi(y)^*
\]
allows one to write elements of $\fF_s$ in terms of
the basis for $\fF_{s+1}$.

It follows that  the range of the expectation $\Phi$ is the
$(mn)^\infty$-UHF algebra $\fF = \ol{\bigcup_{s\ge1} \fF_s}$. This
is a simple C*-algebra.

An almost identical argument is available for the C*-algebra
$\O_{X_u}$. (See also \cite[Proposition 2.1]{Fow}.) As above there
is an abelian group of gauge automorphisms $\gamma_{\alpha,\beta}$
and the map $\Phi :\O_{X_u} \to \O_{X_u}$ is a faithful
expectation onto its range. Moreover the range is equal to the
fixed point algebra, $\O_{X_u}^\gamma$, of the automorphism group
and this can be  identified with a UHF C*-algebra, $\fF_{X_u}$
say. To see this, note that in the universal representation, we have
\begin{align*}
 e_i^*f_j &= e_i^*f_j(\sum_k e_ke_k^*)\\
 &= \sum_k \sum_{i'=1}^m\sum_{j'=1}^n
 \bar{u}_{(i',j'),(k,j)}e_i^*e_{i'}f_{j'}e_k^* \\
 &= \sum_k \sum_{j'=1}^n  \bar{u}_{(i,j'),(k,j)} f_{j'}e_k^*.
\end{align*}
This, as before, leads to the fact that the operators
$\pi(x)\pi(y^*)$, for $x,y \in X_u$, span a dense $*$-algebra in
$\O_{X_u}$. Moreover, the span of
\[
\{\pi(x)\pi(y^*): \ x,y \in E^{\otimes s} \otimes F^{\otimes t},
\ (s,t) \in \bZ^2_+\}
\]
has closure equal to the range of $\Phi$ and, as before, this is a
UHF C*-algebra.

\begin{lem} \label{L:gaugeunitary}
Let $\lambda_\tau$ be a tail representation of $\A_u$. Then the
C*-algebras $\ca(\lambda_\tau(\A_u))$ and $\cenv(\A_u)$ carry
gauge automorphisms which commute with the natural quotient maps
\[
\O_{X_u} \rightarrow \ca(\lambda_\tau(\A_u)) \rightarrow
\cenv(\A_u)
\]
In the case of $\ca(\lambda_\tau(\A_u))$, the gauge automorphisms
are unitarily implemented.
\end{lem}

\begin{proof}
We use the notation of Example~\ref{tail_u}.
Thus $\xi^{s+1}_{we_{i_s}f_{j_s}} = \xi^s_w$ and
\[
\xi^s_{e_uf_v}=\xi^{s+k}_w = \xi^{s+k}_{e_{u'}f_{v'}} ,
\]
where
\[
 w = e_uf_ve_{i_s}f_{j_s}\dots e_{i_s+k-1}f_{j_s+k-1}
 = e_{u'}f_{v'} ;
\]
moreover $|u'| = |u|+k$ and $|v'|=|v|+k.$

Thus we may define a well-defined diagonal unitary
$U_{\alpha,\beta}$ on $\H_\tau$ such that, for $s \geq 0$,
\[
 U_{\alpha,\beta} \xi_{e_uf_v}^{s} =
 \alpha^{|u|-s}\beta^{|v|-s} \xi_{e_uf_v}^{s}.
\]
Now
\[
 U_{\alpha,\beta} \lambda_\tau(e_i)
 U_{\alpha,\beta}^* \xi_{e_uf_v}^{s}
 = \alpha \xi_{e_ie_uf_v}^{s}
 = \alpha \lambda_\tau(e_i) \xi_{e_uf_v}^{s}
\]
and
\[
 U_{\alpha,\beta} \lambda_\tau(f_j)
 U_{\alpha,\beta}^* \xi_{e_uf_v}^{s}
 = \beta \xi_{f_je_uf_v}^{s}
 = \beta\lambda_\tau(f_j) \xi_{e_uf_v}^{s}.
\]
It follows that  $\Ad U_{\alpha,\beta}$ determines an automorphism
of $\lambda_\tau(\A_u)$, denoted also by $\gamma_{\alpha,\beta}$
in view of the gauge action.

These automorphisms are completely isometric, since they are
restrictions of $*$-automorphisms.
So by the universal
property of the C*-envelope, each automorphism has a unique
completely positive extension to $\cenv(\A_u)$ and the extension
is a $*$-isomorphism. In this way a gauge action is determined on
$\cenv(\A_u)$. That the maps commute with the quotients is
evident.
\end{proof}

The next  lemma follows a standard technique in graph C*-algebra.
See \cite[Theorem~3.4]{KumPask} for example.

\begin{lem}
Let $\pi : \O_{X_u} \to B$ be a homomorphism of C*-algebras and
let $\delta : \bT^2 \to \Aut (B)$ be an action such that $\pi
\circ \gamma_{\alpha,\beta} = \delta_{\alpha,\beta} \circ \pi$ for
all $(\alpha, \beta)$ in $\bT^2$. Suppose that $\pi$ is nonzero on
the UHF subalgebra $\fF_{X_u}$. Then $\pi$ is faithful.
\end{lem}

\begin{proof}
As before let $\Phi$ be the expectation map on $\O_{X_u}$, and let
$\Phi_\delta$ the expectation on $B$ induced by $\delta$.
If $\pi(x) = 0$, then
\[ 0=\Phi_\delta(\pi(x^*x)) = \pi(\Phi(x^*x)) .\]
Since $\fF_{X_u}$ is simple  and the restriction of $\pi$ to it is
non zero by assumption it follows that the restriction is
faithful. Thus $\Phi(x^*x) = 0$ and now the faithfulness of $\Phi$
implies $x=0$.
\end{proof}

\begin{thm}\label{Cenv}
The C*-envelope of the unitary relation algebra $\A_u$ is the
generalised Cuntz algebra $\O_{X_u}$ of the product system $X_u$
for the unitary matrix $u$. In particular the C*-envelope of
$\Ath$ is $\ca(\Fth)$. Also each tail representation
$\lambda_\tau$ extends to a faithful representation of $\O_{X_u}$.
\end{thm}

\begin{proof}
This is immediate from the lemma in view of the fact that there is
a quotient map $q$ of $\O_{X_u}$ onto $\cenv(\A_u)$ which commutes
with gauge automorphisms on both algebras.
\end{proof}

\begin{eg}\label{flipCenv}
Consider the flip graph semigroup $\Fth$ of Example~\ref{E:flip}.
Kumjian and Pask observed that $\ca(\Fth) \simeq \O_2 \otimes
\rC(\bT)$. To see this in an elementary way, consider the
relations
\[ e_i f_j = f_i e_j \qforal 1 \le i,j \le 2 .\]
Suppose that $\sigma(e_i)=E_i$ and $\sigma(f_j)=F_j$ is a
$*$-representation. Then $E_i$ and $F_i$ have the same range for
$i=1,2$. Therefore there are unitaries $U_i$ so that $F_i =
E_iU_i$. Then the commutation relations show that
\begin{alignat*}{2}
E_1^2U_1 &= E_1U_1E_1 &\quad E_1E_2U_2 &= E_1U_1E_2\\
E_2E_1U_1 &= E_2U_2E_1 &\quad E_2^2U_2 &= E_2U_2E_2 .
\end{alignat*}
Therefore
\[ E_1U_1=U_1E_1 = U_2E_1 \qand E_2U_2=U_2E_2=U_1E_2 .\]
It follows that $U_1=U_2 =:U$ on $\ran E_1 + \ran E_2 = \H$; and
that $U$ commutes with $\ca(E_1,E_2) \simeq\O_2$.

Consequently an irreducible $*$-representation $\pi$ of
$\ca(\Fth)$ sends U to a scalar $tI$, and the restriction of
$\pi$ to $\ca(e_1,e_2)$ is a $*$-representation of $\O_2$. All
representations of $\O_2$ are $*$-equivalent because $\O_2$ is
simple. Therefore, $\pi(f_i) = t\pi(e_i)$ and $\ca(\pi(\Fth))
\simeq\O_2$. It is now easy to see that
\[ \ca(\Fth) \simeq \O_2 \otimes \rC(\bT) \simeq \rC(\bT,\O_2) .\]
By Theorem \ref{Cenv}, this is also the C*-envelope $\cenv(\Ath)$.
The structure of $\ca (\Ath)$ is will now follow from Lemmas
\ref{I+Jquotient} and \ref{I+Jideals}.
\end{eg}

\bigskip

We can use Theorem \ref{Cenv} and the theory of C*-envelopes and
maximal dilations to identify the completely contractive
representations of $\A_u$ with those that have dilations to defect
free isometric representations, that is, to $*$-representations.
As we note in the next section, the contractive representations of
$\A_u$ form a wider class. First we recap the significance of
maximal dilations.

Recall that a representation $\pi$ of an algebra $\A$, or
semigroup, on a Hilbert space $\K$ is a \textit{dilation} of a
representation $\sigma$ on a Hilbert space $\H$ if there is an
injection $J$ of $\H$ into $\K$ so that $J\H$ is a semi-invariant
subspace for $\pi(\A)$ (i.e.\ there is a $\pi(\A)$-invariant
subspace $\M$ orthogonal to $J\H$ so that $\M \oplus J\H$ is also
invariant) so that $J^* \pi(\cdot) J = \sigma(\cdot)$.

A dilation $\pi$ of $\sigma$ is \textit{minimal} if the smallest
reducing subspace containing $J\H$ is all of $\K$. This minimal
dilation is called unique if for any two minimal dilations $\pi_i$
on $\K_i$, there is a unitary operator $U$ from $\K_1$ to $\K_2$
such that $J_2 = UJ_1$ and $\pi_2 = \Ad U \pi_1$.

\pagebreak[3]
Generally we are interested in dilations within the same class,
such as row contractive representations of semigroups which are
generated by two free families, or completely contractive
representations of algebras. A representation $\sigma$ within a
certain class of representations is called \textit{maximal} if
every dilation $\pi$ of $\sigma$ has the form $\pi\simeq\sigma
\oplus \pi'$, or equivalently $J\H$ always reduces $\pi$. It is
possible for a dilation to be both minimal and maximal.

In his seminal paper \cite{Arv1}, Arveson showed how to understand
non-selfadjoint operator algebras in terms of dilation theory. He
defined the C*-envelope of an operator algebra $\A$ to be the
unique C*-algebra $\cenv(\A)$ containing a completely
isometrically isomorphic copy of $\A$ which generates it, but any
proper quotient is no longer completely isometric on $\A$. He was
not able to show that this object always exists, but that was
later established by Hamana \cite{Ham}. For background on
C*-envelopes, see Paulsen~\cite{Pau}.

A completely contractive unital representation of an operator
algebra $\A \subset \ca(\A)$ has the \textit{unique extension
property} if there is a unique completely positive extension to
$\ca(\A)$ and this extension is a $*$-representation. If this
$*$-representation is irreducible, it is called a \textit{boundary
representation}.

There is a new proof of the existence of the C*-envelope.
Dritschel and McCullough \cite{DMc} showed that the C*-envelope
can be constructed by exhibiting sufficiently many representations
with the unique extension property. Arveson \cite{Arv_choq}
completed his original program by then showing that it suffices to
use irreducible representations.

The insight of Dritschel and McCullough, based on ideas of Agler,
was that the maximal completely contractive dilations coincide
with dilations with the unique extension property. Therefore
maximal dilations factor through the C*-envelope. In particular, a
maximal representation $\sigma$ which is completely isometric
yields the C*-envelope: $\cenv(\A) = \ca(\sigma(\A))$.

 From a different point of view, this was also observed by
Muhly and Solel \cite{MSmem}. They show that a completely
contractive unital representation factors through the C*-envelope
if and only if it is orthogonally injective and orthogonally
projective.  While we do not define these notions here, we point
out that it is easy to see that these two properties together are
equivalent to being a maximal representation.

The upshot of the theory of C*-envelopes and maximal dilations is
the following consequence. Recall that a $*$-representation of
$\bF^+_u$ is a representation satisfying the unitary commutation
relations which is isometric and defect free.

\pagebreak[3]
\begin{thm}\label{DilnAth}
Let $\sigma$ be  a unital representation  $\Fu$ satisfying the
unitary commutation relations. Then the following are equivalent:
\begin{enumerate}
\item $\sigma $ dilates to a $*$-representation of $\Fu$.
\item $\sigma $ is completely contractive, that is, $\sigma$
extends to a completely contractive representation of $\A_u$.
\end{enumerate}
In  particular a unital representation of the semigroup
$\bF_\theta^+$ dilates to a $*$-representation if and only if it
is completely contractive.
\end{thm}

\begin{proof}
Suppose that $\sigma$ dilates to a $*$-dilation $\pi$. By the
definition of $\O_{X_u}$, $\pi$ extends to a $*$-representation of
$\O_{X_u}$. By Theorem~\ref{Cenv}, $\A_u$ sits inside $\O_{X_u}$
completely isometrically. As $*$-representations are completely
contractive, it follows that $\pi$ restricts to a completely
contractive representation of $\A_u$. By compression to the
original space, we see that $\sigma$ is also completely
contractive on $\A_u$.

Conversely, any completely contractive representation $\sigma$ of
$\A_u$ has a maximal dilation $\pi$. Thus it has the unique
extension property, and so extends to a $*$-representation of
$\cenv(\A_u)$. By Theorem~\ref{Cenv}, $\cenv(\A_u) = \O_{X_u}$.
Therefore $\pi$ restricts to a $*$-representation of $\bF^+_u$.
\end{proof}

\subsection{Ideals of the C*-algebra $\ca(\A_u)$}
We shall show that $\O_{X_u}$ is a quotient of $\ca(\A_u)$.
Indeed, there are several ideals that are evident:
\begin{align*}
 \K &:= \big\lip\big( I - \sum_i \lambda(e_i) \lambda(e_i)^* \big)
 \big( I - \sum_j  \lambda(f_j) \lambda(f_j)^* \big)\big\rip\\
 \I &:= \big\lip\big( I - \sum_i \lambda(e_i) \lambda(e_i)^*\big)\big\rip\\
 \J &:= \big\lip\big( I - \sum_j  \lambda(f_j) \lambda(f_j)^*\big)\big\rip\\
 \I+\J &= \big\lip\big( I - \sum_i \lambda(e_i) \lambda(e_i)^* \big),
 \big( I - \sum_j  \lambda(f_j) \lambda(f_j)^*\big)\big\rip .
\end{align*}

Note that the projections $$P =(I - \sum_i \lambda(e_i)
\lambda(e_i)^*) \qand  Q =(I - \sum_j \lambda(f_j)
\lambda(f_j)^*)$$ are the projections onto the subspaces
\[
 \sum_{l=0}^\infty \bC\otimes F^{\otimes l} \qand
 \sum_{k=0}^\infty E^{\otimes k}\otimes \bC
\]
and that $PQ=QP$ is the rank one
 projection $\xi_\mt \xi_\mt^*$.
Note that $\lambda(e_uf_v)) \xi_\mt \xi_\mt^* \lambda(e_sf_t)^*$
is the rank one operator $\xi_{e_uf_v}\xi_{e_sf_t}^*$ mapping
basis element $\xi_{e_sf_t}$ to basis element $\xi_{e_uf_v}$. Thus
a complete set of matrix units for $\L(\H_u)$ is available in
$\K$, and so $\K = \fK$, the ideal of compact operators.

The projection $P$ generates a copy of $\fK$ in
$\ca(\{e_i\}) \simeq \E_m$, where the matrix
units permute the subspaces
\[
 \xi_{e_u}\otimes (\sum_{l=0}^\infty \bC\otimes F^{\otimes l})
 =  \spn\{ \xi_{e_u f_v}: f_v \in \Fn \}.
\]
Also it is clear that $P\A_u P$ is a copy of $\A_n$, the
noncommutative disk algebra generated by $f_1,\dots,f_n$ and so it
generates a copy of the Cuntz--Toeplitz algebra $\E_n$ acting on
$P\H_u$. It is now easy to see that $\I$ is $*$-isomorphic to $\fK
\otimes \E_n$.

Similarly, $\J$ is isomorphic to $\E_m \otimes \fK$.
The intersection of these two ideals is $\I \cap \J = \K$;
and $\K$ is isomorphic to $\fK \otimes \fK$ sitting inside
both $\I$ and $\J$.
Then $\I+\J$ is also an ideal by elementary C*-algebra theory.

\begin{lem}\label{I+Jquotient}
The quotient $\ca(\A_u)/(\I\!+\!\J)$ is isomorphic to $\O_{X_u}$.
\end{lem}

\begin{proof}
The quotient $\ca(\A_u)/(\I\!+\!\J)$ yields a representation of
$\bF^+_u$ as isometries.  It is defect free by construction, and
thus $\ca(\A_u)/(\I\!+\!\J)$ is a quotient of $\O_{X_u}$. It is
easy to see that the gauge automorphisms leave $\I$, $\J$ and $\K$
invariant; and so $\ca(\A_u)/(\I\!+\!\J)$ has a compatible family
of gauge automorphisms. Thus the quotient is again isomorphic to
$\ca(\Fu)$ as in the proof of Theorem~\ref{Cenv}. In particular,
this quotient is completely isometric on $\A_u$.
\end{proof}

\begin{lem}\label{I+Jideals}
The only proper ideals of $\I + \J$ are $\I$, $\J$ and $\K$.
\end{lem}

\begin{proof}
It is a standard result that if a C*-algebra of operators
acting on a Hilbert space contains $\fK$,
then $\fK$ is the unique minimal ideal. So $\K$ is the
unique minimal ideal of $\ca(\A_u)$.

Suppose that $\M$ is an ideal of $\I+\J$ properly containing $\K$.
Then $\M/\K$ is an ideal of
\[ (\I\!+\!\J)/\K \simeq \O_m \otimes\fK \oplus \fK \otimes \O_n .\]
The two ideals $\I/\K \simeq \O_m \otimes\fK$ and
$\J/\K \simeq \fK \otimes \O_n$ are mutually orthogonal and
simple.
So the ideal $\M/\K$ either contains one or the other or both.
\end{proof}

Kumjian and Pask define a notion called the \textit{aperiodicity
condition} for higher rank graphs.
In our context, for the algebra $\A_\theta$ it means
that there is an \textit{irreducible} representation of type 3a.
They show \cite[Proposition~4.8]{KumPask} that aperiodicity
implies the simplicity of $\ca(\Fth)$.
The converse is established by Robertson and Sims \cite{RobSims}.
In \cite{DYperiod}, this is examined carefully.
Aperiodicity seems to be typical, but
there are periodic 2-graphs such as the flip algebra of
Example~\ref{E:flip}.

When $\ca(\Fth)$ is simple, we have described the complete ideal
structure of $\ca(\Ath)$. For the general case, see
\cite{DYperiod}.

\section{Row Contractive Dilations}

Now we turn to dilation theory. We saw in Theorem~\ref{DilnAth}
that maximal completely contractive representations of $\A_u$
correspond to the $*$-representations of $\Fu$. In the next
section, we will show that defect free contractive representations
of $\Fu$ are completely contractive, and therefore dilate to
$*$-representations. Here we consider row contractive
representations and give a simple proof of Solel's result that
they dilate to row isometric representations. Despite such
favourable dilation we give examples of contractive
representations that contrast significantly with the defect free
case.  In particular we show that contractive representations of
$\Fth$ need not be completely contractive.

\begin{eg}\label{flipdefect}
Consider the flip graph of Examples~\ref{E:flip} and \ref{flipCenv}.
Define the representation of $\Ath$
on a basis $\xi_0,\xi_1,\xi_2,\zeta_1,\zeta_2$ given by
\[
\pi(e_i) = \zeta_i \xi_1^* \quad \pi(f_1) = \zeta_1 \xi_0^* \qand
\pi(f_2) = \zeta_2 \xi_2^*.
\]
Note that $\pi$ is row contractive.

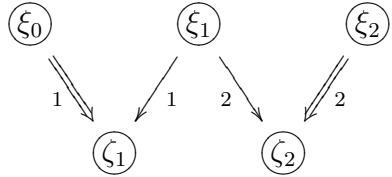
\[
 \xymatrix@!@C=.1in{
  {\xy *++={\xi_0} *\frm{o} \endxy} \ar@{=>}[dr]_-1
&& {\xy *++={\xi_1} *\frm{o} \endxy} \ar[dl]^-1 \ar[dr]_-2 && {\xy
*++={\xi_2} *\frm{o} \endxy} \ar@{=>}[dl]^-2
\\ & {\xy *++={\zeta_1} *\frm{o} \endxy}
&& {\xy *++={\zeta_2} *\frm{o} \endxy} }
\]

However \textit{$\pi$ does not dilate to a defect free isometric
representation.} To see this, suppose that $\pi$ has a dilation
$\sigma$ that is isometric and defect free. The path from $\xi_0$
to $\xi_2$ is given by $\pi(f_2^* e_2 e_1^* f_1)$. However in any
defect free dilation,
\begin{align*}
 \sigma( f_2^* e_2 e_1^* f_1) &= \sigma( f_2^* e_2 e_1^* f_1)\sigma (e_1e_1^* + e_2 e_2^*) \\
                     &= \sigma( f_2^* e_2 e_1^* (e_1 f_1 e_1^* + e_1 f_2 e_2^*)) \\
                     &=  \sigma( f_2^* e_2 (f_1 e_1^* + f_2 e_2^*)) \\
                     &=  \sigma( f_2^* f_2 (e_1 e_1^* + e_2 e_2^*)) \\
                     &=  \sigma( 1 ) = I.
\end{align*}
Hence $\xi_2 = \sigma( f_2^* e_2 e_1^* f_1) \xi_0 = \xi_0$,
contrary to fact.

Next we show that \textit{$\pi$ is contractive on $\Ath$.}
We need to show that $\| \pi(x) \| \le \| \lambda(x) \|$
for $x \in \Ath$. Let
\[ x = a + b_1 e_1 + b_2 e_2 + c_1 f_1 + c_2 f_2 + \text{higher order terms} .\]
Then
\[
 \pi(x) =  \left[\begin{array}{ccc|cc}
 a&0&0&0&0\\
 0&a&0&0&0\\
 0&0&a&0&0\\
 \hline
 c_1&b_1&0&a&0\\
 0&b_2&c_2&0&a
\end{array}\right]
= \begin{bmatrix} aI_3&0\\X&aI_2\end{bmatrix} .
\]
Now the $5\times 5$ corner of $ \lambda(x)$ on
$\spn\{\xi_\mt, \xi_{e_1},\xi_{e_2},\xi_{f_1},\xi_{f_2}\}$ has the form
\[
 \left[\begin{array}{c|cccc}
 a&0&0&0&0\\
 \hline
 b_1&a&0&0&0\\
 b_2&0&a&0&0\\
 c_1&0&0&a&0\\
 c_2&0&0&0&a
\end{array}\right]
= \begin{bmatrix} a&0\\ y&aI_4\end{bmatrix}
\]
Note that $\|X\| \le \|X\|_2 = \|y \|_2$. So
\[
 \|\pi(x)\| \le \left\| \begin{bmatrix} |a|&0 \\ \|X\|&|a|\end{bmatrix}\right\| \le
  \left\| \begin{bmatrix} |a|&0 \\ \|y\| &|a|\end{bmatrix}\right\|  \le \|\lambda(x)\|.
\]

Nevertheless, we show that \textit{$\pi$ is not completely contractive.}
Let $B_1=B_2= \begin{bmatrix}1&0\end{bmatrix}$
and $C_1=-C_2 = \begin{bmatrix}0&1\end{bmatrix}$;
and consider the matrix polynomial $X = B_1e_1+B_2e_2+C_1f_1+C_2f_2$.
Then
\[
 \| \pi(X)\| = \norm{\left[ \begin{array}{cc|cc|cr}0&1&1&0&0&0\\
 \hline 0&0&1&0&0&-1\end{array} \right]} =
 \norm{\begin{bmatrix}1&1&0\\0&1&1\end{bmatrix}} = \sqrt3.
\]
By Example~\ref{flipCenv}, the C*-envelope of $\Ath$ is $\O_2\otimes\rC(\bT)$.
As shown there, an irreducible representation $\sigma$ is determined by its
restriction to $\ca(e_1,e_2)$ and a scalar $t \in \bT$
so that $\sigma(f_i) = t\sigma(e_i)$.
Since $\ca(e_1,e_2) \simeq\O_2$ is simple, it does not matter which
representation is used, as all are faithful.
Let $S_i = \sigma(e_i)$ be Cuntz isometries.
Then the norm $\lambda(X)$ is determined as the supremum over
$t\in\bT$ of these representations.
\begin{align*}
 \| \lambda(X)\| &=
 \sup_{t \in \bT}  \| (B_1+tC_1) \otimes S_1 + (B_2+tC_2)\otimes S_2 \| \\
 &= \sup_{t \in \bT} \norm{\begin{bmatrix}B_1+tC_1\\B_2+tC_2\end{bmatrix}}
 = \sup_{t \in \bT} \norm{\left[ \begin{array}{cr}1&t\\1&-t\end{array} \right]} =
 \sqrt2 .
\end{align*}

An alternative proof is obtained by noting that by Theorem~\ref{DilnAth},
if $\pi$ were completely contractive on $\Ath$, then one could dilate it to a
$*$-representation of $\Fth$, which was already shown to be impossible.
\end{eg}

A related example shows that a row contractive representation
may not even be contractive.

\begin{eg}\label{not contractive}
Take any $\Fth$ for which there are indices $i_0$ and $j_0$ so
that there is no solution to $e_{i_0} f_j = f_{j_0} e_i$. The flip
graph is such an example, with $i_0=1$ and $j_0=2$. Consider the
two dimensional representation $\pi$ of $\Fth$ on $\bC^2$ with
basis $\{\xi_1,\xi_2\}$ given by
\[
 \pi(e_{i_0}) = \pi(f_{j_0}) = \xi_2\xi_1^*
 \qand \pi(e_i)=\pi(f_j) = 0 \text{  otherwise.}
\]
The product $\pi(e_if_j) = 0$ for all $i,j$; so this is a representation.
Evidently it is row contractive.

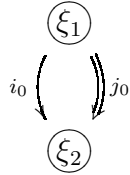
\[
 \xymatrix{
  {\xy *++={\xi_1} *\frm{o} \endxy}
 \ar@/^{-1pc}/[d]_-{i_0} \ar@/^{1pc}/ @{=>}[d]^-{j_0}\\
  {\xy *++={\xi_2} *\frm{o} \endxy}
}
\]

However $\pi(e_{i_0}+f_{j_0}) = 2 \xi_2\xi_1^*$ has norm $2$.
The hypothesis guarantees that no word beginning with $e_{i_0}$
coincides with any word beginning with $f_{j_0}$.
Thus in the left regular representation, $\lambda(e_{i_0})$ and
$\lambda(f_{j_0})$ are isometries with orthogonal ranges.
Hence $\|\lambda(e_{i_0}+f_{j_0})\| = \sqrt2$.

So this row contractive representation does not extend
to a contractive representation of $\Ath$.
\end{eg}

Another problem with dilating row contractive representations is
that the minimal row isometric dilation need  not be unique.
Consider the following illustrations.

\begin{eg}
Let $\pi$ be the 2-dimensional trivial representation of $\Fth$,
$\pi(\mt) = I_2$ and $\pi(w) = 0$ for $w\ne\mt$.
Evidently this dilates to the row isometric representation
$\lambda\oplus\lambda$; and this is clearly minimal.

Now pick any $i,j$ and factor $e_if_j = f_{j'}e_{i'}$.
Inside of the left regular representation, identify
$\bC^2$ with $\M_0 := \spn\{\xi_{e_{i'}}, \xi_{f_j} \}$.
Note that the compression of $\lambda$ to $\M_0$
is unitarily equivalent to $\pi$.
The invariant subspace that $\M_0$ determines
is $\M = \ol{ \Ath \xi_{e_{i'}} + \Ath \xi_{f_j} }$.
The restriction $\sigma$ of $\lambda$ to $\M$ is therefore
a minimal row isometric dilation of $\pi$.
However
\[ \sigma(e_i) \xi_{f_j} = \xi_{e_if_j} = \xi_{f_{j'}e_{i'}} = \sigma(f_{j'})\xi_{e_{i'}} .\]
For any non-zero vector $\zeta = a\xi_{e_{i'}} + b\xi_{f_j}$ in $\M_0$,
either $\sigma(e_i)^*\sigma(f_{j'}) \zeta = a  \xi_{f_j}$ or $\zeta$ itself
is a non-zero multiple of $ \xi_{f_j}$; and similarly $\xi_{e_{i'}}$
belongs to the reducing subspace containing $\zeta$.
Therefore $\sigma$ is irreducible.

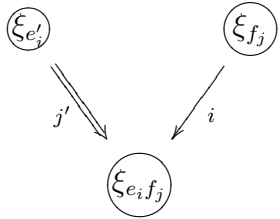
\[
 \xymatrix@!@C=.1in{
  {\xy *++={\xi_{e_i'}} *\frm{o} \endxy}
 \ar@{=>}[dr]_-{j'} &&
  {\xy *++={\xi_{f_j}} *\frm{o} \endxy}
 \ar[dl]^-{i}\\
  &{\xy *++={\xi_{e_if_j}} *\frm{o} \endxy}
}
\]

So these two minimal row isometric dilations are not unitarily equivalent.
\end{eg}

\begin{eg}
Here is another example where the original representation is irreducible.
Consider $\Fth$ where $m=2$, $n=3$ and the permutation $\theta$
has cycles
\[ \big( (1,2) , (2,1) \big) \qand \big( (2,2),(2,3) , (1,3) \big) .\]
Let $\pi$ be the representation on $\bC^3$ with basis
$\zeta_1, \zeta_2, \zeta_3$ given by
\[ \pi(e_1) = \zeta_3 \zeta_1^* \qand \pi(f_1) = \zeta_3 \zeta_2^* \]
and all other generators are sent to $0$.
We show that this may be dilated to a subrepresentation of $\lambda$
in two different ways.

First identify $\zeta_1$ with $\xi_{f_1}$, $\zeta_2$ with
$\xi_{e_1}$ and $\zeta_3$ with $\xi_{e_1f_1} = \xi_{f_1e_1}$. Then
a minimal row isometric dilation is obtained by $\sigma_1 =
\lambda|_{\M_1}$ where $\M_1 = \ol{ \Ath \xi_{e_1} + \Ath
\xi_{f_1} }$. A second dilation is obtained from the
identification of $\zeta_1$ with $\xi_{f_2}$, $\zeta_2$ with
$\xi_{e_2}$ and $\zeta_3$ with $\xi_{e_1f_2} = \xi_{f_1e_2}$. Then
$\sigma_2 = \lambda|_{\M_2}$ where $\M_2 = \ol{ \Ath \xi_{e_2} +
\Ath \xi_{f_2} }$.

These two dilations are different because
\[ \sigma_1(e_2) \xi_{f_1} = \xi_{e_2f_1} = \xi_{f_2e_1} = \sigma_1(f_2) \xi_{e_1} \]
while
\[ \sigma_2(e_2) \xi_{f_2} = \xi_{e_2f_2} \ne \xi_{f_2e_2} = \sigma_2(f_2) \xi_{e_2}. \]
So the two dilations are not equivalent.

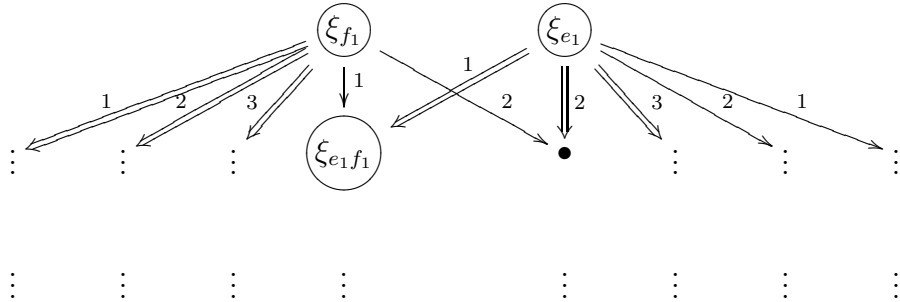
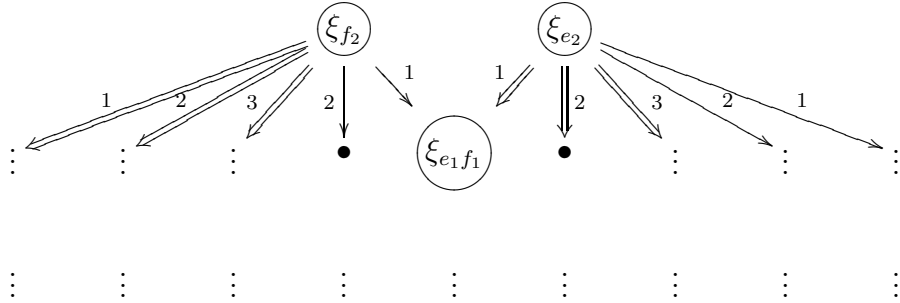
\[
\xymatrix@!@R=.15in@C=.08in{
&&&    {\xy *++={\xi_{f_1}} *\frm{o} \endxy} 
  \ar@{=>}[dlll]_(.7)1 \ar@{=>}[dll]_(.7)2 \ar@{=>}[dl]_(.7)3 
  \ar[d]^(.4)1 \ar[drr]^(.7)2&& 
  {\xy *++={\xi_{e_1}} *\frm{o} \endxy} 
  \ar@{=>}[dll]_(.4)1 \ar@{=>}[d]^-2 
\ar@{=>}[dr]^(.7)3 \ar[drr]^(.7)2 \ar[drrr]^(.7)1  
\\ \vdots&\vdots&\vdots&
  {\xy *++={\xi_{e_1f_1}} *\frm{o} \endxy} 
&& \bullet
&\vdots&\vdots&\vdots
\\ \vdots&\vdots&\vdots&\vdots  
&&\vdots&\vdots&\vdots&\vdots&
}
\]

\[
\xymatrix@!@R=.15in@C=.08in{
&&&  {\xy *++={\xi_{f_2}} *\frm{o} \endxy} 
\ar@{=>}[dlll]_(.7)1 \ar@{=>}[dll]_(.7)2 \ar@{=>}[dl]_(.7)3 
  \ar[dr]^-1 \ar[d]_-2&& 
  {\xy *++={\xi_{e_2}} *\frm{o} \endxy} 
 \ar@{=>}[dr]^(.7)3 \ar[drr]^(.7)2 \ar[drrr]^(.7)1  
  \ar@{=>}[dl]_-1 \ar@{=>}[d]^-2 
\\ \vdots&\vdots&\vdots&\bullet
&  {\xy *++={\xi_{e_1f_1}} *\frm{o} \endxy}
& \bullet
&\vdots&\vdots&\vdots
\\ \vdots&\vdots&\vdots&\vdots  
&\vdots&\vdots&\vdots&\vdots&\vdots&
}
\]
\end{eg}

With these examples as a caveat, we provide a simple proof of
Solel's result  \cite[Corollary~4.5]{SolCP}. Our proof is based on
the much more elementary result of Frahzo \cite{Fra}, Bunce
\cite{Bun} and Popescu \cite{Pop} that every contractive $n$-tuple
has a unique minimal dilation to a row isometry.

First we recall some details of Bunce's proof. Consider a row
contraction $A =
\begin{bmatrix}A_1 & \dots &A_m
\end{bmatrix}$. Following Schaeffer's proof of Sz.\ Nagy's
isometric dilation theorem, let $D_A = (I_{\bC^m \otimes \H} -
A^*A)^{1/2}$. Observe that $\begin{bmatrix}A\\D_A \end{bmatrix}$
is an isometry. Hence the columns $\begin{bmatrix}A_i\\D_A^{(i)}
\end{bmatrix}$ are isometries with pairwise orthogonal ranges in
$\B(\H, \H\oplus \V\otimes\H)$ where $\V = \bC^m$. Now consider
$\K = \V\otimes\H\otimes\Fockm$ where we identify $\V\otimes\H$
with $\V\otimes\H\otimes\bC\xi_\mt$ inside $\K$. Let $\lambda$ be
the left regular representation of $\Fm$ on $\Fockm$, and set $L_i
= \lambda(e_i)$. Define isometries on $\H \oplus\K$ by
$S_i = \begin{bmatrix}A_i & 0\\
\begin{bmatrix} D_A^{(i)}\\0 \end{bmatrix} &
I_{\V\otimes\H} \otimes L_i \end{bmatrix}$. These isometries have
the desired properties except minimality. One  can then restrict
to the invariant subspace $\M$ generated by $\H$. Popescu
establishes the uniqueness of this minimal dilation in much the
same way as for the classical case.

\begin{lem}\label{L:row_diln}
Let $S= \begin{bmatrix}S_1&\dots&S_m\end{bmatrix}$
be a row isometry, where each $S_i \in \B(\H \oplus \K)$
is an isometry that leaves $\K$ invariant.
Suppose that there is a Hilbert space $\W$ so that
$\K \simeq \W \otimes \Fockm$ and $S_i|_\K \simeq I_\W \otimes L_i$
for $1 \le i \le m$.
Let $\M$ be the smallest invariant subspace for $\{S_i\}$ containing $\H$.
Then $\M$ reduces $\{S_i\}$ and there is a subspace
$\W_0 \subset\W$ so that $\M^\perp \simeq \W_0 \otimes \Fockm$.
\end{lem}

\begin{proof}
Clearly $\M = \bigvee_{w \in \Fm} S_w \H$.
For any non-trivial word $w = iw'$ in $\Fm$,
$S_j^* S_w\H = \delta_{ij} S_{w'}\H$;
and $S_j^* \H \subset\H$ because $\K=\H^\perp$
is invariant for $S_j$.
So $\M$ reduces each $S_j$.

Thus $\M^\perp \subset \K \simeq \W \otimes \Fockm$
reduces each $S_i|_\K \simeq I_\W \otimes L_i$.
But $W^*(L_1,\dots,L_m) = \B(\Fockm)$ because
$\ca(L_1,\dots,L_m)$ contains the compact operators.
Hence $W^*(\{S_i|_\K\}) \simeq \bC I_\W \otimes \B(\Fockm)$.
Therefore a reducing subspace is equivalent to one of
the form $\W_0 \otimes \Fockm$.
\end{proof}

\begin{thm}[Solel] \label{T:Solel}
Let $\sigma$ be a row contractive representation of $\Fu$ on $\H$.
Then $\sigma$ has a dilation to a row isometric representation
$\pi$ on a Hilbert space $\H \oplus \K$.
\end{thm}

\begin{proof}
Start with a Hilbert space $\W = \V \otimes \H$, where $\V$ is a
separable, infinite dimensional Hilbert space, and set $\K = \W
\otimes \H_u$. Let $\lambda$ denote the left regular
representation of $\Fu$ on $\H_u$. Note that the restriction to
$\Fm= \ip{e_1,\dots,e_m}$ yields a multiple of the left regular
representation of $\Fm$.

Following Bunce's argument, set $A_i = \sigma(e_i)$ and define
isometries on $\H \oplus\K$ by
$S_i = \begin{bmatrix}A_i & 0\\
\begin{bmatrix} D_A^{(i)}\\0 \end{bmatrix} &
I_{\V\otimes\H} \otimes \lambda(e_i) \end{bmatrix}$.
However, note that the increased size of $\V$ means that
the $m$ element column $D_A^{(i)}$ must be extended by
zeros even within the subspace $\W \otimes \bC \xi_\mt$.
Thus there is always a subspace orthogonal to the
minimal invariant subspace $\M$ containing $\H$ on which
$S_i$ acts like a multiple of the left regular representation
with multiplicity at least $\max\{ \aleph_0,\Dim\H\}$.

Similarly, set $B_j = \sigma(f_j)$ for $1 \le j \le n$, and define
the defect operator $D_B = (I_{\bC^n \otimes \H} - B^*B)^{1/2}$.
Then define isometries on $\H \oplus\K$ by
\[
 T_j = \begin{bmatrix}B_j & 0\\
\begin{bmatrix} D_B^{(j)}\\0 \end{bmatrix} &
I_{\V\otimes\H} \otimes \lambda(f_j) \end{bmatrix} .
\]

Now notice that  in $\bC[\Fu]$ the semigroup generated by $e_1f_1,
\dots, e_mf_n$ is the free semigroup $\Fmn$. Indeed, if $e_if_jw =
e_kf_lw'$, then by cancellation,
it follows that $i=k$, $j=l$ and $w=w'$.  So, with successive cancellation,  the
alternating products $e_if_jw$, $e_kf_lw'$ are equal in $\bC[\Fu]$
only if they are identical.

We will consider two row isometric representations of $\Fmn$:
\[
 \pi_1(e_if_j) = S_iT_j \qand
 \pi_2(e_if_j) = \sum_{i'=1}^m \sum_{j'=1}^n u_{(i,j),(i',j')} T_{j'} S_{i'}
\]
for all $1 \le i \le m$ and $1 \le j \le n$. The reason that
$\pi_2$ has the desired properties is that a $p$-tuple of
isometries with orthogonal ranges spans a subspace isometric to a
Hilbert space consisting of scalar multiples of isometries. So the
fact that $u$ is a unitary matrix ensures that the $mn$ operators
$\pi_2(e_if_j)$ are indeed isometries with orthogonal ranges.

Since $\sigma$ is a representation of $\Fu$, we see that $\pi_1$
and $\pi_2$ both compress to $\sigma$ on $\H$. So both are
dilations of the same row contractive representation of $\Fmn$. By
Lemma~\ref{L:row_diln}, for both $k=1,2$, we have $\pi_k(e_if_j)
\simeq \mu(e_if_j) \oplus I_{\W_k} \otimes \lambda(e_if_j)$ where
$\mu$ is the minimal row isometric dilation of $\sigma|{\Fmn}$ and
$\Dim\W_k = \max\{ \aleph_0,\Dim\H\}$. The two minimal dilations
are unitarily equivalent via a unitary which is the identity on
$\H$, and the multiples of the left regular representation are
also unitarily equivalent. So $\pi_1$ and $\pi_2$ are unitarily
equivalent on $\H\oplus\K$ via a unitary $W$ which fixes $\H$,
i.e.
\[
 \pi_2(e_if_j) = W \pi_1(e_if_j) W^*
 \qforal 1 \le i \le m \AND 1 \le j \le n .
\]

Now set
\[
 \pi(e_i) = S_i W \qand \pi(f_j) = W^* T_j
 \qfor 1 \le i \le m \AND 1 \le j \le n .
\]
This provides a row isometric dilation of $\big[e_1\ \dots\ e_m\big]$
and $\big[f_1\ \dots\ f_n\big]$.  Moreover,
\begin{align*}
 \pi(e_i) \pi(f_j) &= S_iWW^*T_j = S_iT_j =  \pi_1(e_if_j) \\
 &= W^* \pi_2(e_if_j) W =
 \sum_{i'=1}^m \sum_{j'=1}^n u_{(i,j),(i',j')} W^* T_{j'} S_{i'} W \\
 &= \sum_{i'=1}^m \sum_{j'=1}^n u_{(i,j),(i',j')} \pi(f_{j'}) \pi(e_{i'}) .
\end{align*}
So $\pi$ yields a representation of $\Fu$.
\end{proof}

We remark that the case $n=1$ of Solel's theorem was obtained
earlier by Popescu \cite{Pop2}; and the special case of this for
commutant lifting is due to Muhly and Solel \cite{MStensor}.

Our discussion of the flip algebra in Examples \ref{flipdefect} and \ref{not contractive}
show that a row contractive representation of the algebra $\A_u$
need not be contractive. As a consequence,
the natural map $ \T_+(X_u) \to \A_u$ from the tensor algebra is
not isometric.

In fact in this case, using results from \cite{PowSol},
we can show that there is no map which is an isometric isomorphism.
Firstly, note that the explicit unitary
automorphisms of $\A_u$ given there may be readily defined on the
tensor algebra. Secondly, the character space $M(\A_u)$ of $\A_u$
and its core subset (which is definable in terms of nest
representations) identify with the character space and core of
$\T_+(X_u)$. Suppose now that $\Gamma : \A_u \to \T_+(X_u)$ is a
isometric isomorphism. Composing with an appropriate
automorphism of $\T_+(X_u)$, we may assume that the induced
character space map $\gamma$ maps the origin to the origin (in the
realisation of $M(\A_u)$ in $\bC^{n+m}$ \cite{KP1}).  By the
generalized Schwarz inequality in \cite{PowSol}, it follows that
the biholomorphic map $\gamma$ is simply a rotation automorphism,
defined by a pair of unitaries $A \in M_m(\bC)$ and $B \in
M_n(\bC)$. Composing $\Gamma$ with the  inverse of the associated
gauge automorphism $\pi_{A,B}$ of $\T_+(X_u)$, we may assume that
$\gamma$ is the identity map. Since $\Gamma$ is isometric it
follows, as in \cite{PowSol}, that $\Gamma$ is the natural map,
which is a contradiction.

For $n=1$, note that $ \T_+(X_u)$  is the crossed product algebra
$\A_m \times_\alpha \bZ_+$, which is defined as the universal
operator algebra for covariant representations $(\rho, T)$, where
$\rho : \A_m \to \B(\H_\rho)$ is a contractive representation
determined by a row contraction $[S_1 \dots S_m]$ which
$u$-commutes with a contraction $T$. Here $\alpha$ is a gauge
automorphism of $\A_m$ determined by $u$. Moreover in this case
the tensor algebra \textit{is} isometrically isomorphic to $\A_u$.
One way to see this is to note that if ~~$[S_1' \dots S_m'], T'$~~
is an isometric dilation, determined by Popescu-Solel dilation,
then we may apply a second such dilation to ~~$[(S_1')^* \dots
(S_m')^*], T'^*$~~ to derive an isometric dilation of a covariant
representation $(\rho, T)$ of the form $(\sigma , U)$ with $U$
unitary. Such representations are completely contractive on
$\A_u$.

\section{Dilation of defect free Representations}

We now show the distinctiveness of defect free contractive
representations in that they are completely contractive and have
\textit{unique} minimal $*$-dilations. Moreover, we show that
atomic contractive defect free representations of $\Fth$ have
unique minimal atomic representations. This is an essential tool
for the representation theory of 2-graph semigroups developed in
\cite{DPYatomic} because we frequently describe
$*$-representations by their restriction to a cyclic coinvariant
subspace.

\begin{thm}\label{defectfreediln}
Let $\sigma$ be a defect free, row contractive representation of
$\Fu$. Then $\sigma$ has a unique minimal $*$-dilation.

\end{thm}

The proof follows from Theorem \ref{T:Solel} and the next two
lemmas and the fact that a defect free row isometric representation
is a $*$-dilation.

\begin{lem}\label{L:diln stays df}
Let $\sigma$ be a defect free, row contractive representation.
Then any minimal row isometric dilation is defect free.
\end{lem}

\begin{proof}
Let  $\pi$ be a minimal row isometric dilation acting on $\K$. Set
$\M = (I - \sum_i \pi(e_i)\pi(e_i)^*) \K$. We first show that $\M$
is coinvariant. Indeed, if $x \in \M$ and $y \in \K$, then plainly
\begin{align*}
 \ip{ \pi(e_i)^* x, \pi(e_k) y} &= \ip{x, \pi(e_i)\pi(e_k) y}= 0
\end{align*}
for each $i$ and $k$, while, using the commutation relations,
\begin{align*}
 \ip{ \pi(f_l)^* x, \pi(e_k) y} &= \ip{x, \pi(f_le_k) y}\\
& = \sum_{i=1}^m\sum_{j=1}^n {u}_{(i,j),(k,l)}\ip{x,
\pi(e_i)\pi(f_j)y}=0.
\end{align*}
So $ \sigma(w)^* x$ belongs to
$\M = \big( \sum_i \pi(e_i) \K \big)^\perp$ for any word $w$.

If we write each $\pi(e_i)$ as a matrix with respect to $\K=\H
\oplus\H^\perp$, we have $\pi(e_i) =
\begin{bmatrix}\sigma(e_i)&0\\ \ast&\ast\end{bmatrix}$. Therefore
\[
 \sum_i \pi(e_i) \pi(e_i)^* =
 \begin{bmatrix}\sum_i \sigma(e_i)\sigma(e_i)^* &\ast\\ \ast&\ast\end{bmatrix}
 = \begin{bmatrix}I_\H &\ast\\ \ast&\ast\end{bmatrix} .
\]
This is a projection, and thus
\[
 \sum_i \pi(e_i) \pi(e_i)^* =
 \begin{bmatrix}I_\H &0 \\ 0&\ast\end{bmatrix}
 \ge P_\H.
\]
Thus $\M$ is orthogonal to $\H$. It now follows that for any
$x\in\M$, $h \in \H$ and $w \in \Fu$,
\[ \ip{\pi(w)h,x} = \ip{h,\pi(w)^*x} = 0\]
because $\pi(w)^*x \in \M$. But the vectors of the form
$\pi(w)h$ span $\K$, and so $\M=\{0\}$.
\end{proof}

An immediate consequence of this lemma and Theorem~\ref{DilnAth} is:

\begin{cor}
Every defect free, row contractive representation $\pi$ of $\Fu$
extends to a completely contractive representation of $\A_u$.
\end{cor}

\begin{lem}
The minimal row isometric dilation of a defect free, row
contractive representation of $\Fu$ is unique up to a unitary
equivalence that fixes the original space.
\end{lem}

\begin{proof}
 Let $\pi$ be a minimal row isometric
dilation of $\sigma$ on the Hilbert space $\K$. Let $\W$ be the
set of words $w=e_uf_v$ in $\Fu$. By minimality and the
commutation relations, a dense set in $\K$ is given by the vectors
of the form $\sum_k \pi(w_k) h_k$ where this is a finite sum, each
$h_k \in\H$ and $w_k \in \W$. We first show that given any two
such vectors, $\sum_k \pi(w_k) h_k$ and $\sum_l \pi(w'_l) h'_l$,
we may suppose that each $w_k$ and $w'_k$ has the same degree.

To this end, let $d(w_k) = (m_k,n_k)$ and $d(w'_l) = (m'_l,n'_l)$, and set
\[ m_0 = \max\{m_k, m'_l\} \qand n_0=\max\{n_k, n'_l\} .\]
For each $w_k$, let $a_k = m_0-m_k$ and $b_k = n_0-n_k$.
Then because $\pi$ is defect free by Lemma~\ref{L:diln stays df},
\begin{align*}
 \pi(w_k) h_k &=
 \pi(w_k) \Big( \sum_{d(v)=(a_k,b_k)} \pi(v)\pi(v)^* \Big) h_k\\
 &= \sum_{d(v)=(a_k,b_k)}  \pi(w_kv) ( \sigma(v)^* h_k) .
\end{align*}
The second line follows because $\H$ is coinvariant for
$\pi(\Fu)$, and so $\pi(v)^* h_k = \sigma(v)^* h_k$ belongs to
$\H$. Using the commutation relations we may write the original
sum with new terms, each of which has degree $(m_0,n_0)$. Combine
terms if necessary so that the words $w_k$ are distinct. Then we
obtain a sum of the form $\sum_{d(w)=(m_0,n_0)} \pi(w) h_w$. We
similarly rewrite
\[ \sum_l \pi(w'_l) h'_l = \sum_{d(w)=(m_0,n_0)} \pi(w) h'_w .\]

Now the isometries $\pi(w)$ for distinct words of degree
$(m_0,n_0)$ have pairwise orthogonal ranges. Therefore we compute
\begin{align*}
 \big\langle \!\!\!\!\! \sum_{d(w)=(m_0,n_0)} \!\!\!\!\! \pi(w) h_w ,
 \sum_{d(w)=(m_0,n_0)}\!\!\!\!\! \pi(w) h'_w\big\rangle &=
 \sum_{d(w)=(m_0,n_0)}\!\!\!\!\! \ip{ \pi(w) h_w, \pi(w) h'_w} \\&=
 \sum_{d(w)=(m_0,n_0)}\!\!\!\!\! \ip{ h_w,h'_w} .
\end{align*}

Now suppose that $\pi'$ is another minimal row isometric dilation
of $\sigma$ on a Hilbert space $\K'$. The same computation is
valid for it. Thus we may define a map from the dense subspace
$\spn\{\pi(\Fu)\H\}$ of $\K$ to  the dense subspace
$\spn\{\pi'(\Fu)\H\}$ of $\K'$ by
\[
 U \sum_{d(w)=(m_0,n_0)} \pi(w) h_w =
 \sum_{d(w)=(m_0,n_0)} \pi'(w) h_w .
\]
The calculation of the previous paragraph shows that $U$ preserves
inner products, and thus is well defined and isometric. Hence it
extends by continuity to a unitary operator of $\K$ onto $\K'$.
Moreover, each vector in $h$ has the form $h = \pi(\mt)h$; and
thus $Uh=h$.  That is, $U$ fixes the subspace $\H$. Finally, it is
evident from its definition that $\pi'(w) = U \pi(w)U^*$ for all
$w\in\Fu$. So $\pi'$ is equivalent to $\pi$.
\end{proof}

For our applications in \cite{DPYatomic}, we need the
following refinement for atomic representations.

\begin{thm}
If $\sigma$ is an atomic defect free partially isometric
representation of $\Fth$, then the unique minimal 
$*$-dilation $\pi$ is also atomic.
\end{thm}

This follows from the theorem above and the next lemma.

\begin{lem}
Let $\sigma$ be an atomic, defect free, partially isometric
representation of $\Fth$. Then any minimal $*$-dilation
of $\sigma$ is atomic.
\end{lem}

\begin{proof}
Let  $\pi$  be a minimal row isometric dilation of $\sigma$ acting
on $\K$. Consider the standard basis $\{ \xi_k : k \ge1 \}$ for
$\H$ with respect to which $\sigma$ is atomic.
Let $\dot\xi_k$ denote $\bC^*\xi =  \{\alpha\xi_k:
\alpha \in \bC\backslash \{0\}\}.$ We claim that the set $\{
\pi(x) \dot\xi_k : k \ge 1, x \in \Fth\}$ forms an orthonormal
family of 1-dimensional subsets spanning $\K$, with repetitions.
Indeed, $\H$ is coinvariant and cyclic; so these sets span $\K$.
It suffices to show that any two such sets, say $\pi(x_1)
\dot\xi_1$ and $\pi(x_2)\dot\xi_2$, either coincide or are
orthogonal.

Let $d(x_k) = (m_k,n_k)$ for $k=1,2$; and set
\[ (m_0,n_0) =  (m_1,n_1) \vee  (m_2,n_2) = \big(\max\{m_1,m_2\}, \max\{n_1,n_2\} \big) .\]
Since $\sigma$ is defect free, there are unique basis vectors
$\zeta_k$ and words $y_k$ with $d(y_k) = (m_0-m_k,n_0-n_k)$ so
that $\sigma(y_k)\dot\zeta_k = \dot\xi_k$. Thus using
$\dot\zeta_k$ and the word $x_ky_k$, we may suppose that the two
words have the same degree. For convenience of notation, we
suppose that this has already been done.

Write $x_k = e_{u_k}f_{v_k}$.
As noted in the proof of Theorem~\ref{T:Solel}, two distinct
words of the same degree have pairwise orthogonal ranges.
Thus if $x_1 \ne x_2$, then $\pi(x_1)\dot\xi_1$
and $\pi(x_2)\dot\xi_2$ are orthogonal.
On the other hand, if $x_1=x_2$, then if
$\dot\xi_1=\dot\xi_2$, the images are equal; while if
$\dot\xi_1$ and $\dot\xi_2$ are
orthogonal, they remain orthogonal under the action of the
isometry $\pi(x_1)$.
\end{proof}




\end{document}